\documentclass[12pt]{amsart}
\usepackage{graphicx}                                    
\usepackage{amssymb}
\newtheorem{theorem}{Theorem}[section] 
\newtheorem{lemma}[theorem]{Lemma}
\newtheorem{conjecture*}[theorem]{Conjecture}
\newtheorem{proposition}[theorem]{Proposition}

\newtheorem{corollary}[theorem]{Corollary}

\newtheoremstyle{notauto}{}{}{\itshape}{}{\bfseries}{.}{0.5em}{\thmnote{#3}}
\theoremstyle{notauto}

\theoremstyle{definition}
\newtheorem{definition}[theorem]{Definition}
\newtheorem{example}[theorem]{Example}
\theoremstyle{remark}
\newtheorem{remark}[theorem]{Remark}

\begin{document}

\baselineskip20pt

\title[Finite groups having nonnormal T.I. subgroups] 
{Finite groups\\ having nonnormal T.I. subgroups }
\author{{M.Yas\.{I}r} K{\i}zmaz }
\address{Department of Mathematics, Middle East Technical University, Ankara 06531, Turkey}
\email{yasir@metu.edu.tr}
\thanks{This work has been supported by T\"{U}B\.{I}TAK-B\.{I}DEB 2211/A}
\subjclass[2010]{20D10, 20D20}
\keywords{T.I. subgroup, Frobenius group, Double Frobenius group, normal complement}
\maketitle

\begin{abstract} In the present paper, the structure of a finite group $G$ having a nonnormal T.I. subgroup $H$ which is also a Hall $\pi$-subgroup is studied. As a generalization of a result due to Gow, we prove that $H$ is a Frobenius complement whenever $G$ is $\pi$-separable. This is achieved by obtaining the fact that Hall T.I. subgroups are conjugate in a finite group. We also prove two theorems about normal complements one of which generalizes a classical result of Frobenius.
\end{abstract}

\section{Introduction}

In a group $G$, a subgroup $H$ is called a \textit{ T.I. subgroup} if $H\cap H^x =H$ or trivial for any $x\in G$. Clearly, any normal subgroup is also a T.I. subgroup but it is not very interesting beside normality. Hence, most of the time we assume that $H$ is a nonnormal T.I. subgroup.

Throughout the article, we assume that all groups are finite. In one of his celebrated works, Frobenius proved that a self normalizing T.I. subgroup $H$ of a group $G$ has a normal complement in $G$ by using his theory of induced characters. In that case, $H$ is called a Frobenius complement. Moreover, such groups $G$ are named as Frobenius groups after him and they are well studied. 

If a group $H$ acts on a group $K$ by automorphisms with the property that every nonidentity element of $H$ acts on $K$ fixed point freely, the action of $H$ on $K$ is called as \textit{Frobenius}. In that case, the group $KH$ is a Frobenius group with complement $H$, and so we also call $H$ as a Frobenius complement if there exists a group on which $H$ has a Frobenius action.

A result due to Burnside is that Sylow subgroups of Frobenius complements are cyclic or generalized quaternion. Notice that Frobenius complements are Hall subgroups of Frobenius groups, and so there are two natural questions to ask: 

\textit{Assume that $H$ is a Hall subgroup which is also a nonnormal T.I. subgroup of the group $G$.}\\
\textbf{Question A } \textit{Under what conditions is $H$ a Frobenius complement ?}\\
\textbf{Question B } \textit{Under what conditions does $H$ have a normal complement in $G$ ?}\\

In \cite{1} Gow obtained a partial answer to Question A, namely, he proved the following:

\begin{theorem}(Gow) \label{Gow's Theorem} Let $H$ be a Hall subgroup of the solvable group $G$ and suppose that $H$ is a nonnormal T.I. subgroup of G. Then H has an irreducible representation on some elementary abelian section of $G$ in which each of its nonidentity elements acts without fixed points.
\end{theorem}

 In the present paper, we obtain an extension of his result to $\pi$-separable groups as a more general answer to Question A. We prove the following:

\begin{theorem}
    \label{Gen. of Gow}
Let $H$ be a nonnormal T.I. subgroup of the $\pi$-separable group $G$ where $\pi$ is the set of primes dividing the order of $H$. Further assume that $H$ is a Hall subgroup of $N_G(H)$. Then the following hold:\\
    $a)$ $G$ has  $\pi$-length $1$ where $G=O_{\pi'}(G)N_G(H)$;\\
    $b)$ there is an $H$-invariant section of $G$ on which the action of $H$ is Frobenius. This section can be chosen as a chief factor of $G$ whenever $O_{\pi'}(G)$ is solvable;\\
    $c)$ $G$ is solvable if and only if $O_{\pi'}(G)$ is solvable and $H$ does not involve a subgroup isomorphic to $SL(2,5)$.
\end{theorem}

    Parts $(a)$ and $(b)$ of Theorem \ref{Gen. of Gow} not only generalize Theorem \ref{Gow's Theorem} but also determine the structure of the group $G$ under the weaker hypothesis that $H$ is a Hall subgroup of $N_G(H)$. Yet, it turns out to be equivalent to assuming that $H$ is a Hall subgroup of $G$. (See Lemma \ref{lemma: norm}). Here is an example showing that the condition of $\pi$-separability is indispensable. %where $\pi$ is the prime set of $H$
    
    \begin{example}
        Let $K=A_5$ and $H$ be a Sylow $5$-subgroup of $K$. Clearly, $H$ is a T.I. subgroup of $K$ and a Hall subgroup of $N_K(H)$. The $H$-invariant subgroups of $K$ are only $N_K(H)$, $H$ and the trivial subgroup where $|N_K(H)|=10$. Hence, it is easy to see that both $(a)$ and the first part of Theorem \ref{Gen. of Gow} $(b)$  fail to be true.
        \end{example}
        We next present an example which shows that the second part of Theorem \ref{Gen. of Gow} (b) is not true in case $O_{\pi'}(G)$ is nonsolvable even if $G$ satisfies other hypotheses.
    \begin{example}

    Let $N=SL(2,2^7)$ and $\alpha\in Aut(N)$  of order $7$ which arises from the nontrivial automorphism of the field with $2^7$ elements. Set $G=N \langle \alpha \rangle$. Then $\langle \alpha \rangle$ is a nonnormal Hall T.I. subgroup of $G$. Since $N$ is a simple, $\langle \alpha \rangle$ does not have Frobenius action on any chief factor of $G$, and so the second part of Theorem \ref{Gen. of Gow} $(b)$ fails to be true.    \end{example}

We also studied the structure of groups having a nonnormal T.I. subgroup which is also a Hall subgroup, and obtained the following result about the conjugacy of such Hall subgroups. This proposition plays a crucial role in simplifying many proofs, throughout the article.

\begin{proposition}
\label{proposition: key pro}
    
    Let $G$ be a group containing a Hall $\pi$-subgroup $H$ which is also a T.I. subgroup. Then any $\pi$-subgroup of $G$ is contained in a conjugate of $H$. In particular, the set of all Hall $\pi$-subgroups of $G$ forms a single $G$-conjugacy class.
    
\end{proposition}

In the following normal complement theorem we obtained more information on the structure of a $\pi$-separable group $G$ having a nonnormal T.I. subgroup.

\begin{theorem}
    \label{complement theorem}
    Assume that the hypothesis of Theorem \ref{Gen. of Gow} holds. Assume further that a Sylow $2$-subgroup of $H$ is abelian and $Q$ is a complement of $H$ in $N_G(H)$. Then $C_H(Q)$ is a Hall subgroup of $G$ having a normal complement in $G$.
\end{theorem}

This theorem is obtained by using the result below which is of independent interest too.
\begin{proposition}
        \label{proposition: Fusion proposition}
        
        Let $A$ be a group acting coprimely on $G$ by automorphisms. Assume that Sylow subgroups of $G$ are cyclic. Then,\\
        $a)$ $C_G(A)$ is a Hall subgroup of $G$;\\
        $b)$ $G=[G,A]\rtimes C_G(A)$;\\
        $c)$ the group $[G,A]$ is cyclic.
\end{proposition}
    
After analyzing the structure of groups having a nonnormal T.I. subgroup which is also a Hall subgroup, we present the following theorem as a full answer to Question B  by finding a necessary and sufficient condition for Hall T.I. subgroups to have a normal complement. 
This result appears to be a nice application of Theorem \ref{Gen. of Gow} and Proposition \ref{proposition: key pro}. 

\begin{theorem}
    \label{Gen.Frobenius}
    
    Let $H$ be T.I. subgroup of $G$ which is also a Hall subgroup of $N_G(H)$. Then $H$ has a normal complement in $N_G(H)$ if and only if $H$ has a normal complement in $G$. Moreover, if $H$ is nonnormal in $G$ and $H$ has a normal complement in $N_G(H)$ then $H$ is a Frobenius complement.
\end{theorem}

The theorem above also generalizes the classical result of Frobenius which asserts that for a  Frobenius group $G$ with complement $H$, the set $N=(G- \bigcup_{g\in G} H^g)\cup \{1\}$ is a normal subgroup of $G$ with $G=NH$.\\

Notation and terminology are standard as in [2].

    \section{preliminaries}
The following well known results about coprime action will be frequently used throughout the paper without further reference.

\begin{theorem}[Coprime Action (\cite{2}, Chapters 3,4)]
    Let $A$ be a group acting by automorphisms on a group $G$. Assume that $(|A|,|G|)=1$ and $A$ or $G$ is solvable. Then the following hold:\\
    $a)$ The equality $G=[G,A]C_G(A)$ holds. Moreover, If $G$ is abelian, $G=[G,A]\times C_G(A)$.\\
    $b)$ The identity $[G,A,A]=[G,A]$ holds regardless of solvability of $A$ or $G$.\\
    $c)$ For each prime $p$ dividing the order of $G$, there exists an $A$-invariant Sylow $p$-subgroup of $G$.\\
    $d)$ For any $A$-invariant $P\in Syl_p(G)$, $P\cap C_G(A)\in Syl_p(C_G(A))$. \\
    $e)$ For any $A$-invariant normal subgroup $N$ of $G$, $C_G(A)N/N= C_{G/N}(A)$.\\
    $f)$ Whenever any two elements of $C_G(A)$ are conjugate in $G$, they are also conjugate in $C_G(A).$
\end{theorem}

Next, we state a theorem to give transfer theoretical facts pertaining to the proof of Proposition 1.6. They will also be used without further reference.
\begin{definition}
    Let $G$ be a group. Assume that $H\leq K\leq G$. We say that $K$ controls $G$-fusion in $H$ if whenever two elements of $H$ are conjugate in $G$, they are also conjugate in $K$. 
\end{definition}

\begin{theorem}[\cite{2}, Chapter 5]
    Let $G$ be a group and $P$ be a Sylow $p$-subgroup of $G$ for a prime $p$ dividing the order of $G$. Then the following hold:\\
    $a)$ If $P\leq Z(N_G(P))$ then $G$ has a normal $p$-complement. \\
    $b)$ If $P$ is cyclic and $p$ is the smallest prime diving the order of $G$ then $G$ has a normal $p$-complement. \\
    $c)$ $G$ has a normal $p$-complement if and only if $P$ controls the $G$-fusion in $P$. 
    
\end{theorem}

We close this section with a character theoretical result, due to Brauer and Suzuki, to which we appeal in the proof of Theorem \ref{Gen.Frobenius}.

\begin{theorem}[\cite{3}, Theorem 8.22]
    Let $H$ be a Hall subgroup of $G$ and suppose that whenever two elements of $H$ are conjugate in $G$, they are already conjugate in $H$. Assume that for every elementary subgroup $E$ of $G$, if $|E|$ divides ${|H|}$, then $E$ is conjugate to a subgroup of $H$. Then $H$ has a normal complement in $G$.
\end{theorem}

 \section{Proofs of the Key Propositions}
 
 \begin{proof}[Proof of Proposition \ref{proposition: key pro}]
     Let $K$ be a $\pi$-subgroup of $G$ which is not contained in any conjugate of $H$. Let $P \in Syl_p(K)$ for a prime $p$ dividing the order of $K$. It should be noted that Sylow $p$-subgroups of $H$ are also Sylow $p$-subgroups of $G$, and hence there exists $x\in G$ such that $P\leq H^x$.
     Set $T= H^{x} \cap K$. It is straightforward to see that $T$ is a T.I. subgroup of $K$. Note that $T$ is nontrivial as $P$ is contained in $T$.
     Pick an element $n$ from $N_{K}(T)$. Then $ T\leq H^{x}\cap H^{xn}$ and so $n$ normalizes $H^x$. This forces that $n\in H^{x}$, because otherwise the group $H^x \langle n \rangle$ contains $H^{x}$ properly.
     It follows that $T$ is a self normalizing T.I. subgroup of $K$, and hence $T$ is a Hall subgroup of $K$. 
     
     Let now $q$ be a prime dividing $\mid$$K$:$T$$\mid$ and pick $Q\in Syl_{q}(K)$. A similar argument as above shows that $Q\leq H^{y} \cap K$ for some $y\in G$. Set $S=H^y\cap K$. Clearly, the group $S$ is also a self normalizing T.I. subgroup of $K$. 
     If $T\cap S^k\neq 1$ for some $k\in K$, then  $H^x=H^{yk}$ as their intersection is nontrivial. This forces the equality $T=S^k$ which is not possible as $q$ is coprime to the order of $T$. Thus we have $T\cap S^k =1$ for all $k\in K$.
     As a consequence we get $S \subseteq(K -\bigcup_{k\in K} T^k)\cup \{1\}$, and hence
     $$\bigcup_{k\in K} S^k \subseteq(K -\bigcup_{k\in K} T^k)\cup \{1\}$$ 
     
     A simple counting argument shows that
     $$ |K|-\dfrac{|K|}{|S|}+1\leq \dfrac{|K|}{|T|},$$
     
     and so,
     $$1\leq \dfrac{1}{|S|}+\dfrac{1}{|T|}.$$
     
     As $T$ and $S$ are both nontrivial, we have $|S|=|T|=2$. That is $q=2$ as $Q \leq S$. Then $q$ also divides the order of $T$, which is a contradiction. Thus, $K$ is contained in a conjugate of $H$, in particular, if $K$ is a Hall $\pi$-subgroup of $G$ then $K\leq H^g$ for some $g\in G$. Consequently, we have $K=H^g$ completing the proof.
 \end{proof}

\begin{proof}[Proof of Proposition \ref{proposition: Fusion proposition}]
    $a)$ Let $p$ a prime dividing the order of $C_G(A)$ and $P$ be an $A$-invariant Sylow $p$-subgroup of $G$. Then $C_P(A)\in Syl_p(C_P(A))$ and $P=[P,A]\times C_P(A)$ due to the coprimeness. As $C_P(A)$ is nontrivial and $P$ is cyclic, we get $C_P(A)=P$. Hence $p$ is coprime to the index of $C_G(A)$ in $G$.\\
    $b)$ It suffices to show that $C=C_G(A)\cap [G,A]=1.$ Assume the contrary and let $p$ the smallest prime dividing the order of $C$. Let $P\in Syl_p(C)$. By $a)$, $P$ is also a Sylow subgroup of $[G,A]$. Burnside's  Theorem implies that $P$ has a normal $p$-complement in $C$, equivalently $P$ controls $C$-fusion in $P$. Note that $C$ controls $[G,A]$-fusion in $C$, and hence, $P$ controls $[G,A]$-fusion in $P$. Thus, $[G,A]$ has a normal $p$-complement $N$, that is, $[G,A]=NP$. This leads to $[G,A,A]\leq N < [G,A]$, which is a contradiction. Thus, $C=C_G(A)\cap [G,A]=1$ as claimed.\\
    $c)$ It is enough to show that $K=[G,A]$ is nilpotent. We can assume that $K$ is nontrivial. Note that $A$ acts on $K$ fixed point freely by $b)$.
    Let $P$ be an $A$-invariant Sylow $p$-subgroup of $K$ for an arbitrary prime $p$ dividing the order of $K$. Clearly, $N_K(P)$ is also an $A$-invariant subgroup of $K$ so that $N_K(P)A$ acts on $P$ by automorphisms. Since $Aut(P)$ is abelian, $[N_K(P),A]$ acts trivially on $P$. As $A$ acts on $N_K(P)$ fixed point freely we get $[N_K(P),A]=N_K(P)$. Then $P\leq Z(N_K(P))$ and so $K$ is $p$-nilpotent. Since $p$ is arbitrary, $K$ is nilpotent as claimed.
\end{proof}

\section{Some Technical lemmas}

In this section, we present four lemmas which will be frequently used in proving the main results. The first one is a generalization of the fact that a self normalizing T.I. subgroup is also a Hall subgroup.

\begin{lemma}
    \label{lemma: norm}
    
    Let $H$ be a T.I. subgroup of a group $G$. Then  $H$ is a Hall subgroup of $G$ if and only if $H$ is a Hall subgroup of $N_G(H)$.
    
\end{lemma}

\begin{proof}
    One direction is trivial to show. Let $H$ be a Hall subgroup of $N_G(H)$, and let $p$ be a prime dividing both  $|H|$ and $|G:H|$. Pick $P\in Syl_p(H)$, $Q\in Syl_p(G)$ such that $P< Q$, and $x\in N_Q(P)-P$. As $H\cap H^x$ is nontrivial, $x \in N_G(H)$. This forces that $x\in H$ since $H$ is a Hall subgroup of $N_G(H)$. Then $P \langle x \rangle$ is a $p$-subgroup of $H$ containing $P$ properly. This contradiction completes the proof.
\end{proof}

\begin{remark}
    Note that the lemmas below can be proven by using the conjugacy part of Schur-Zassenhaus theorem and Feit-Thompson's odd order theorem. Here, we get the same conclusion without appealing to the odd order theorem. Lemma \ref{lemma:p-invariant} is well known yet Lemma \ref{lemma: normalizer} and Lemma \ref{lemma: image of T.I} first time appears as far as author is aware.
    \end{remark}

\begin{lemma}
    \label{lemma:p-invariant}
    
    Let $G$ be a group and $N$ be a normal subgroup of $G$ with a complement $H$. If $(|H|,|N|)=1$ and $H$ is a T.I. subgroup of $G$ then for each prime $p$ dividing the order of $N$, there is an $H$-invariant Sylow $p$-subgroup of $N$.
\end{lemma}

\begin{proof}
    By the Frattini argument, $N_G(P)N=G$ for any $P\in Syl_p(N)$. We have $K=N\cap N_G(P)$ is normal in $N_G(P)$ and $|N_G(P):K|=|G:N|$. Then $K$ is coprime to its index in $N_G(P)$. By the existence part of Schur-Zassenhaus theorem, $N_G(P)=KU$ where $|U|=|G:N|=|H|$.  It follows by Proposition \ref{proposition: key pro} that $U=H^g\leq N_G(P)$ for some $g\in G$, that is, $H$ normalizes $P^{g^{-1}}$ as required.
\end{proof}

\begin{lemma}
    \label{lemma: normalizer}
    
    Let $N$ be a normal subgroup of $G$ and $H$ be a T.I. subgroup of $G$ with $(|H|,|N|)=1$. Then $N_{\overline G}(\overline H)=\overline {N_G(H)}$ where $\overline G=G/N$.
\end{lemma}    
\begin{proof}
    It is clear that $\overline {N_G(H)}\leq N_{\overline G}(\overline H) $. Let $X$ be the full inverse image of $N_{\overline G}(\overline H)$ in $G$. Then $N_G(H)$ is contained in $X$ and $HN\trianglelefteq X$. By Proposition \ref{proposition: key pro}, every complement of $N$ in $HN$ is a conjugate of $H$. Then Frattini argument implies the equality $N_X(H)N=X$. As $N_X(H)=N_G(H)$, we have $X=N_G(H)N$. Then $\overline X= \overline {N_G(H)}$ as claimed.
\end{proof}

It should be noted that a homomorphic image of a T.I. subgroup need not be T.I. subgroup of the image. Therefore, the following lemma provides a sufficient condition for which image of a T.I. subgroup is also T.I. subgroup.
\begin{lemma}
    \label{lemma: image of T.I}
    
    Let $H$ be a T.I. subgroup of $G$ and $N$ be a normal subgroup of $G$ with $(|N|,|H|)=1$. Then $HN/N$ is a T.I subgroup of $G/N$.
\end{lemma}

\begin{proof}
    Set $T=HN\cap H^gN$ for $g\in G$ then $T=T\cap HN=T\cap H^gN$. As $N\leq T$, we have 
    $T=N(T\cap H)=N(T\cap H^g)$. Note that $(|N|,|T\cap H|)=1$ and $T\cap H$ is a T.I. subgroup of $T$. Then the complements $T\cap H$ and $T\cap H^g$ are conjugate by an element of $N$ by Proposition \ref{proposition: key pro}, that is, $(T\cap H)^n=T\cap H^n=T\cap H^g$
    for some $n\in N$. Thus, $T\cap H^n=T\cap H^n\cap H^g=1$ or $H^n=H^g$. If the former holds, then  $T=N$. If the latter holds, then $T=HN$.
    
    So, we have either $T=HN$ or $N$. As a result, $HN/N$ is a T.I. subgroup of $G/N$ because $HN/N \cap H^gN/N=T/N$ is either trivial or equal to $HN/N$.
\end{proof}

 \section{$\pi$-separable groups having nonnormal Hall T.I. subgroups }
 
 This section is devoted to the proofs of Theorem \ref{Gen. of Gow} and Theorem \ref{complement theorem}.
 
 \begin{proof}[Proof of Theorem \ref{Gen. of Gow}]
     
     $a)$ Note that $H$ is a Hall subgroup of $G$ by Lemma 2.1. Then we have $O_{\pi}(G)=1$ as $H$ is a nonnormal T.I. subgroup of $G$. This forces that $O_{\pi'}(G)\neq 1$ since $G$ is $\pi$-separable. Set $\overline G=G/O_{\pi'}(G)$. Notice that $\overline H$ is a T.I subgroup of $\overline G$ by Lemma \ref{lemma: image of T.I}. \\
     Assume that $\overline H$ is not normal in $\overline G$.
     Then $O_\pi(\overline G)$ is trivial. On the other hand $O_{\pi'}(\overline G)$ is also trivial, and so $\overline G$ is trivial which is not the case. 
     Thus, $\overline H$ is normal in $\overline G$. So the lower $\pi$-series of $G$ is as follows;
     $$1<O_{\pi'}(G)<HO_{\pi'}(G)\leq G.$$
     Now $\overline G=N_{\overline G}(\overline H)=\overline{N_G(H)}$ by Lemma \ref{lemma: normalizer}. Therefore, we have $G=O_{\pi'}(G)N_G(H)$ as claimed.

     $b)$ Let $G$ be a minimal counterexample to part $(b)$. Note that if $H\leq K<G$, we may assume that $H \trianglelefteq K $ by the minimality of $G$. Notice that for any $H$-invariant subgroup $R$ of $O_{\pi'}(G)$ with $RH\neq G$, we have $[R,H]=1$ as $H\lhd RH$. In particular, If $HO_{\pi'}(G)\ne G$ then $[H,O_{\pi'}(G)]=1$. As $O_\pi(G)=1$, we have $C_G(O_{\pi'}(G))\leq O_{\pi'}(G)$, and so $H\leq O_{\pi'}(G)$ which is a contradiction. Hence the equality $G=O_{\pi'}(G)H$ holds. Lemma \ref{lemma:p-invariant} guaranties the existence of an $H$-invariant Sylow $p$-subgroup $P$ of $O_{\pi'}(G)$ for any prime $p$ dividing the order of $O_{\pi'}(G)$. If $P\neq O_{\pi'}(G)$, then $[P,H]=1$. Since $p$ is arbitrary, we obtain $[H,O_{\pi'}(G)]=1$ which is impossible. Thus, $G=PH$ where $P=O_{\pi'}(G)$. 
     
 If $[P,H]\neq P$, then $[P,H,H]=1$. Due to coprimeness, we have $[P,H]=1$ which is not the case. Thus $[P,H]=P$, and hence $C_{P/P'}(H)=1$ because $P/P'=[P/P',H]$. Note that $H$ is a T.I. subgroup of $\overline G=(P/P')H$ by Lemma \ref{lemma: image of T.I}. It follows that $\overline G$ is a Frobenius group as $N_{\overline G}(H)=C_{P/P'}(H)H=H$, that is, $H$ is a self normalizing T.I. subgroup of $\overline G$, completing the proof of the first part of $(b)$. 
     
 Suppose next that $O_{\pi'}(G)$ is solvable. Set $L=[O_{\pi'}(G),H]$. It is almost routine to check that action of $H$ on $L/L'$ is Frobenius by using Lemma \ref{lemma: image of T.I}. Note that $L$ is normalized by both $N_G(H)$ and $O_{\pi'}(G)$, and so $L$ is normal in $G=O_{\pi'}(G)N_G(H)$. Therefore, we observed that any chief factor of $G$ between $L$ and $L'$ is a chief factor on which the action of $H$ is Frobenius.\\
 $c)$ Assume that $O_{\pi'}(G)$ is solvable and $H$ does not contain a subgroup isomorphic to a $SL(2,5)$. As $H$ is a Frobenius complement by part (b), we observe that $H$ is solvable. Set $\overline G= G/O_{\pi'}(G)$. Now $\overline{G}=\overline{N_G(H)}= \overline H \rtimes \overline{Q}$ where $N_G(H)=HQ$ by part (a). Therefore, it suffices to show that $\overline Q$ is solvable. Note that $\overline Q$ acts faitfully on $\overline H$ as $C_{\overline G}(\overline H)\leq \overline H$. Due to coprimeness, for each $p\in \pi$, there exists a $\overline Q$-invariant Sylow $p$-subgroup $P$ of $\overline H$. Since $P$ is either cyclic or generalized quaternion, the group $\overline Q/C_{\overline Q}(P)$ is solvable. It follows now that $\overline Q\cong \overline Q/(\bigcap_{p\in \pi} C_{\overline Q}(P))$ is solvable, establishing the claim.\\
     \end{proof}
     \begin{remark}
 Under the hypothesis of Theorem \ref{Gen. of Gow}, we have $G=O_{\pi'}(G)N_G(H)$. On the other hand, by Schur-Zassenhaus theorem, $H$ has a complement in $N_G(H)$, say $Q$. Set $O=O_{\pi'}(G)$. Then we have the equality $G=OHQ$. Note that this need not be a semidirect product as $O\cap Q$ may not be trivial. From now on, we decompose $G$ as $OHQ$ for simplicity whenever the hypothesis holds. 
 
 \iffalse   Let $K$ be a double Frobenius group. Then $K=(A\rtimes B)\rtimes C$  where $AB$ and $BC$ are Frobenius group. $B$ is a nonnormal T.I. subgroup of $K$ and $K=AN_K(B)$ as $N_K(B)=BC$. \fi 
 
 The structure of $G$, in somehow, resembles the structure of a double Frobenius group. More precisely, the following theorem shows that there is a factor group of $G$ containing double Frobenius groups under some additional hypothesis.
     \end{remark}

    \begin{theorem}
        \label{doubleFrob}
      Assume that the hypothesis of Theorem \ref{Gen. of Gow} hold. Assume further that $H$ is of odd order with $[O,H]=O$ and that $O$ is solvable with $Q\nleq O'$. Set $\overline G=G/O'$. Then\\
       $a)$ $\overline{G}=(\overline O \rtimes \overline{H})\rtimes \overline{Q} $;\\
       $b)$  $\overline Q$ is an abelian group acting faithfully on  $\overline H$; \\
       $c)$ $\overline{O}{[\overline H,\beta]}\langle \beta \rangle$ is a double Frobenius group for every element $\beta \in \overline Q$ of prime order.
\end{theorem}
\begin{proof}
         $a)$ It is enough to show that $\overline O \cap \overline Q=1$. Note that $C_{\overline{O}}(\overline H) =1$  as $[\overline{O},\overline{H}]=\overline{[O,H]}=\overline{O}$. For $x\in \overline Q\cap \overline O$, we have $[x,\overline H]\leq \overline  H\cap \overline O =1$. Hence, $x\in C_{\overline{O}}(\overline H)=1$ and $(a)$ follows.
         
         $b)$ Notice that the action of ${\overline Q \overline O}/ \overline O$ on  ${\overline H \overline O}/ \overline O$ is faithful as ${\overline H \overline O}/ \overline O = O_{\pi}(\overline G/\overline O)$. Then $\overline Q$ acts faithfully on $\overline H$, and so the action of $\overline Q$ on $\overline [\overline H ,\overline Q]$ is also faithful. 
         Every Sylow subgroup of $\overline H$ is cyclic because $H$ is a Frobenius complement of odd order. By Proposition \ref{proposition: Fusion proposition}, we get $[\overline H ,\overline Q]$ is cyclic, and hence $\overline Q$ is abelian.
         
         $c)$ Notice that $\overline Q$ is nontrivial. Let $\beta\in \overline Q$ of prime order. $[\overline H,\beta]$ is nontrivial since $\overline Q$ is faithful on $\overline H$. The group $\overline O \overline H$ is Frobenius as $\overline H$ is a self normalizing T.I. subgroup of $\overline O \overline H$. Then $\overline O [\overline H,\beta]$ is also a Frobenius group. By Proposition \ref{proposition: Fusion proposition}(b), $\beta$ acts fixed point freely on $[\overline H,\beta]$, and so $\overline{O}{[\overline H,\beta]}\langle \beta \rangle$ is a double Frobenius group as claimed.
\end{proof}
 
 We close this section by the proof of Theorem \ref{complement theorem}.
 
 \begin{proof}[Proof of Theorem \ref{complement theorem}]
     $H$ is a Frobenius complement by Theorem \ref{Gen. of Gow}(b), and so every Sylow subgroup of $H$ is cyclic as Sylow $2$-subgroup of $H$ is abelian. It follows by Proposition \ref{proposition: Fusion proposition} that $C_H(Q)$ is a Hall subgroup of $H$. Then $C_H(Q)$ is also a Hall subgroup of $G$. Set $N=O[H,Q]Q$. Clearly $N$ is a group. Proposition \ref{proposition: Fusion proposition} yields that $C_H(Q)$ and $[H,Q]$ have coprime orders, and hence  $C_H(Q)$ and $N$ have coprime orders. Then $C_H(Q)\cap N=1$. We also observe that $C_H(Q)$ normalizes $N$ and  $C_H(Q) N=G$ as $G=OHQ$. Hence, $N$ is the desired normal complement for $C_H(Q)$ in $G$. 
 \end{proof}
 
 \section{A Generalization of Frobenius' Theorem}
 
 In this section, we first prove Theorem \ref{Gen.Frobenius} which gives a full answer to Question B. We should note that this theorem can be proven by using the idea of Brauer-Suzuki theorem for the character theoretic part, instead we directly appeal to Brauer-Suzuki theorem for character theoretic part. Frobenius Theorem directly follows from this theorem as self normalizing T.I. subgroup naturally satisfies the hypothesis of the theorem.

 \begin{proof}[Proof of Theorem \ref{Gen.Frobenius}]
     It is trivial to show that if $H$ has a normal complement in $G$ then $H$ has a normal complement in $N_G(H)$.\\
     Assume now that $H$ has a normal complement $Q$ in $N_G(H)$. Then $N_G(H)=Q H$ and $[Q,H]=1$. We claim first that $H$ controls $G$-fusion in $H$: Let $x$ and $x^g$ be elements of $H$ for nonidentity $x\in H$ and for some $g\in G$. Now $x\in H\cap H^{g^{-1}}$, and so $H=H^{g^{-1}}$ as $H$ is a T.I. subgroup. Thus, $g\in N_G(H)$, that is, $g=sh$ for $s\in Q$ and $h\in H$. It follows that $x^g=x^{sh}=x^h$ establishing the claim. Note that by Lemma \ref{lemma: norm}, $H$ is a Hall $\pi$-subgroup of $G$  for the prime set $\pi$ of $H$. Then every elementary $\pi$-groups is contained in a conjugate of $H$ by Proposition \ref{proposition: key pro}. Now appealing to Brauer-Suzuki Theorem we see that $G$ has a normal $\pi$-complement.\\
 Now assume $H$ has normal complement in $N_G(H)$ and $N_G(H)<G$. By the argument above, $G$ has a normal $\pi$-complement. It follows that $G$ is $\pi$-separable, and hence $H$ is a Frobenius complement by $(b)$ part of Theorem \ref{Gen. of Gow}.
 \end{proof}
 
 \begin{corollary}[Frobenius]
     Let $H$ be a proper subgroup of $G$ with $H\cap H^x=1$ for all $x\in G-H$. Then the set $N=(G- \bigcup_{g\in G} H^g)\cup \{1\}$ is a normal subgroup of $G$ where $N$ is a complement for $H$ in $G$.

 \end{corollary}

 \section*{Acknowledgements}
 I would like to thank G\"{u}lin ERCAN and \.{I}smail \c{S}. G\"{U}LO\u{G}LU for their invaluable support and helpful comments.


\begin{thebibliography}{9}
     \bibitem{1} R.Gow, Some T.I. subgroups of solvable groups. J. London Math. Soc. (2) 12 (1975/76), no. 3, 285286.
     \bibitem{2} I.M. Isaacs, Finite group theory. Graduate Studies in Mathematics, 92. American Mathematical Society, Providence, RI, 2008.
     
     \bibitem{3} I.M. Isaacs, Character theory of finite groups. Corrected reprint of the 1976 original [Academic Press, New York; MR0460423]. AMS Chelsea Publishing, Providence, RI, 2006.
 \end{thebibliography}
\end{document}